\documentclass[11pt]{amsart}

\usepackage{enumitem}
\usepackage[utf8]{inputenc}
\usepackage{graphicx}

\usepackage{subfigure}
\usepackage{float}
\usepackage{graphicx}
\usepackage{amsmath,amsthm,amssymb,latexsym,amscd} 

\usepackage{mathrsfs,esint} 

\newtheorem{theorem}{Theorem}[section]

\newtheorem{corollary}[theorem]{Corollary}
\newtheorem{proposition}[theorem]{Proposition}

\theoremstyle{definition}

\theoremstyle{remark}

\numberwithin{equation}{section}

\newcommand{\M}{\mathcal{M}}
\newcommand{\N}{\mathcal{N}}

\begin{document}

\title{Words in Random Binary Sequences I}


\author{}\address{Corresponding Author: Frank Neubrander \\ Department of Mathematics\\ Louisiana State University\\ Baton Rouge, Louisiana 70803,  fneubr1@lsu.edu}
\author{}\address{Other Authors: (a) Christian Ennis \\ Department of Mathematics\\ Louisiana State University\\  Baton Rouge, LA\\ cennis4@lsu.edu\\ (b) William Holland, Alabama School of Fine Arts, Birmingham, AL \\ williamholland28@gmail.com\\ (c) Omer Mujawar\\Johns Creek High School\\ Johns Creek, GA \\omerrayhan@gmail.com\\ (d) Aadit Narayanan, Episcopal High School, Baton Rouge\\ LA\\  narayanan.aadit@gmail.com \\ (e) Marie Neubrander, Department of Mathematics, University of Alabama, Tuscaloosa, AL\\ mneubrander@crimson.ua.edu\\(f) Christina Simino, Baton Rouge, LA,  christinasimino@gmail.com }
\subjclass[2010]{05A15, 60C05 \\ {\it Acknowledgement: This paper originated with the project ``Fibonacci and the Gambler's Ruin'' that was offered at Louisiana State University's 2020 Virtual Summer Math Circle program for high school students. The authors thank the LSU Math Circle program for this  opportunity.}}
\maketitle
\noindent \begin{center} CHRISTIAN ENNIS, WILLIAM HOLLAND, OMER MUJAWAR, AADIT NARAYANAN, FRANK NEUBRANDER, MARIE NEUBRANDER,  AND CHRISTINA SIMINO\end{center}

\section{Introduction}

\noindent
Consider a game in which one flips a balanced (fair) coin until one gets a given word like $HH$ (two heads in a row) or $HT$ (a head followed by a tail). The less flips it takes, the more one will win. Which word would be better to bet on: $HH$ or $HT$? 

\smallskip
\noindent
 More generally, let $W=L_1L_2...L_N$ with $L_i\in\{H,T\}$ be a binary word of length $N$.  In this paper we determine the number of ways $a_{W}(n)$ that binary words $W$ of lengths $N=2$ or $N=3$ can appear {\it for the first time} after $n$ coin tosses, where $n \in \mathbb{N} = \{1,2,3,...\}$. The essential two-letter words are $HH$ and $HT$; by replacing all $H$'s with $T$'s and vice versa, one obtains the words $TT$ and $TH$. 

\smallskip \noindent
As a first example, consider the word $W=HT$. When flipping a coin once, there are $2$ possible outcomes: $H$ and $T$. Therefore, $a_{HT}(1)=0$. When flipping a coin twice, there are $2^2$ possible outcomes: $HH,HT,TH,TT$. Therefore, $a_{HT}(2)=1$. Finally, when flipping a coin three times, there are $2^3$ possible outcomes: $HHH$, $HHT$, $HTH$, $HTT$, $TTT$, $TTH$, $THT$, $THH$. Therefore, $a_{HT}(3)=2$. 

\begin{figure}[htp]
    \centering
    \includegraphics[width=12cm]{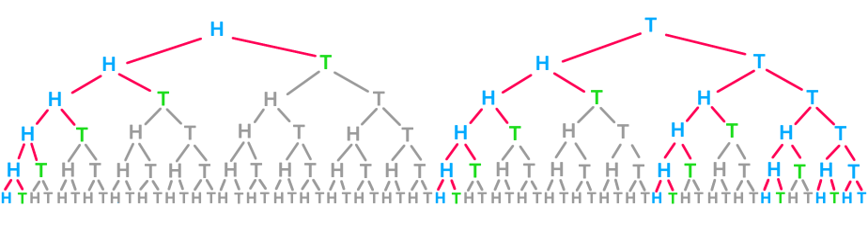}
    \caption{$a_{HT}(n)$ for $1\le n \le 6$ are given by $0,1,2,3,4,5$}

\end{figure}
\noindent
As we will prove in Theorem 2.1, $a_{HT}(n)=n-1$ for all  $n\in\mathbb{N}$. As a consequence, we will prove that one can expect to flip a coin four times until one gets---for the first time---the word $HT$. 

\newpage

\noindent
As a second example, consider the word $W=HH$. 

\begin{figure}[htp]
    \centering
    \includegraphics[width=12cm]{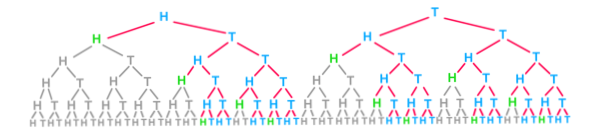}
    \caption{$a_{HH}(n)$ for $1\le n \le 6$ are given by $0,1,1,2,3,5$}
    \label{fig:galaxy}
\end{figure}
\noindent As we will prove, the sequence $a_{HH}(n)$ is indeed the Fibonacci sequence; i.e., $a_{HH}(n+2) = a_{HH}(n+1) + a_{HH}(n)$ for all $ n\ge 1$, with $a_{HH}(1) = 0$ and $a_{HH}(2) = 1$.  We use this to show that one can expect to flip a coin six times until one gets---for the first time---the word $HH$ (see also \cite{HH}). Thus, in the aforementioned game, it is better to bet on $HT$ than on  $HH$. Also, the gambler should know that in about $10\%$  of trials it takes $7$ or more coin flips until one receives a head followed by a tail and  $12$ or more coin flips until one receives a head followed by a head (see Corollaries 2.2 and 2.3).

\smallskip\noindent
In Section 3 of this paper, we will prove recursion formulas of the type
$$a_{W}(n+3) = Aa_{W}(n+2) + Ba_{W}(n+1)+ C a_{W}(n)$$ for all three-letter words. The numbers $A,B,C$ as well as the first 15 terms of the sequences $a_W(n)$ are summarized in the following table.

\smallskip

\begin{center}
\begin{tabular}{ |c|c|c|c|c| } 
 \hline
  W & A, B, C & Sequence $a_W(n)$ for $1\le n\le 15$  \\ 
  \hline
 HHH & 1, ${}$ 1, ${}$ 1    & 0, 0, 1, 1, 2, 4, ${}$ 7, 13, 24, 44, 81, 149, 274, 504, 927\\ 
 HTT & 2, ${}$ 0, -1 & 0, 0, 1, 2, 4, 7, 12, 20, 33, 54, 88, 143, 232, 376, 609 \\ 
 HHT & 2, ${}$ 0, -1 & 0, 0, 1, 2, 4, 7, 12, 20, 33, 54, 88, 143, 232, 376, 609 \\ 
 HTH & 2, -1, ${}$ 1 & 0, 0, 1, 2, 3, 5, ${}$ 9, 16, 28, 49, 86, 151, 265, 465, 816 \\
 
 \hline
\end{tabular}
\end{center}

\smallskip\noindent
Observe that the formulas for $TTT$, $THH$, $TTH$, and $THT$ are identical to the ones in the table above by replacing all $H$'s with $T$'s and vice versa.  

\smallskip\noindent
Part II of this paper (\cite{EHMN}) contains recursive formulas for $a_W(n)$ for binary words $W$  of arbitrary length $N$. That is, for words with $N$ letters, there are constants $A_i$ ($1\le i \le N$) such that 
$$a_{W}(n+N) = \sum_{i=1}^N A_i a_W(n+N-i)$$
with $a_W(k) = 0$ for all $1\le k\le N-1$ and $a_W(N) = 1$.

\bigskip \noindent\noindent
The probabilities  $p_{W}(n)$ that binary words $W$ appear {\it for the first time} after $n$ coin tosses are given by
\begin{equation} \label{prob}
p_W(n) = a_W(n)\left(\frac{1}{2}\right)^n.
\end{equation}
Then,
$$ p_W(\le m) := \sum_{n=1}^m p_W(n) =  \sum_{n=1}^m a_W(n) \left(\frac{1}{2}\right)^n
$$
is the probability that  a binary word $W$ appears {\it for the first time} during the first $m$ coin tosses and 
\begin{equation*}
p_W(\ge m+1) = 1-p_W(\le m) 
\end{equation*}
is the probability that it takes {\bf at least} $m+1$ coin tosses for the binary word $W$ to appear {\it for the first time}.
To study the probabilities $p_W$ and to determine the expected value 
$\mu_W = \hbox{Exp}_W = \sum_{n=1}^\infty n\, p_W(n)$
as well as the variance 
$\hbox{Var}_W = \sum_{n=1}^\infty n^2\, p_W(n) - \mu_W^2$
and standard deviation
$\sigma_W = \sqrt{\hbox{Var}_W}$
of coin flips needed so that a binary word $W$ appears {\it for the first time}, we consider the generating functions  
\begin{equation} \label{gf} f_W(x) = \sum_{n=1}^\infty a_W(n) \, x^n
\hbox{ and }f_{W,m}(x) = \sum_{n=1}^m a_W(n) x^n.
\end{equation}
Since $1\le a_W(n)\le 2^n$ for $n\ge 1$, it follows that the radius of convergence $R(f_W)$ of the power series $f_W$ satisfies $\frac{1}{2} \le R(f_W) \le 1$.
\begin{proposition} \label{Prop1} Let $p_W(n)$, $f_{W}(x)$, and $f_{W,m}(x)$ be defined as in (\ref{prob}) and (\ref{gf}). Then $p_W(\le m) = f_{W,m}(1/2)$. Moreover, if $R(f_W) > \frac{1}{2}$, then
\smallskip
\begin{itemize}
    \item[(a)] $\mu_w = \frac{1}{2}f'_W(\frac{1}{2})$,
    \item[(b)] $\sigma_W = \sqrt{\hbox{Var}_W}$, where $\hbox{Var}_W = \frac{1}{4}f''_W(\frac{1}{2}) + \mu_W -  \mu_W^2.$
\end{itemize}
\end{proposition}
\begin{proof}
Clearly,
$p_W(\le m) = \sum_{n=1}^m p_W(n) =  \sum_{n=1}^m a_W(n) (\frac{1}{2})^n = f_{W,m}(1/2).$
Because 
$xf'_W(x) = \sum\limits_{n=1}^\infty n \, a_W(n) \, x^n, $ it follows that 
$$ \mu_w = \hbox{Exp}_W = \sum_{n=1}^\infty n\, p_W(n) =   \sum_{n=1}^\infty n\, a_W(n) (\frac{1}{2})^n = \frac{1}{2}f'_W(\frac{1}{2}).$$
This shows (a). To prove (b), observe that 
$$xf'_W(x) + x^2 f''_W(x) = x(xf'_W(x))' = \sum_{n=1}^\infty n^2 \, a_W(n) \, x^n.$$  This shows that
\begin{align*}
 \hbox{Var}_W  &=  \sum_{n=1}^\infty n^2\, p_W(n) - \mu_W^2  = \sum_{n=1}^\infty n^2\, a_W(n)(\frac{1}{2})^n - \mu_W^2   \\
 &= \frac{1}{2}f'_W(\frac{1}{2}) + \frac{1}{4}f''_W(\frac{1}{2}) -  \mu_W^2 
 = \frac{1}{4}f''_W(\frac{1}{2}) + \mu_W -  \mu_W^2.
\end{align*}
\end{proof}
\goodbreak

\section{Two-Letter Words}

\begin{theorem}
When flipping a fair coin,  let $a_{W}(n)$ be the number of ways that the word $W = HT$ or $W = HH$ can appear for the first time at the $n$-th coin toss. Then,  
\begin{equation}  \label{HT} a_{HT}(n+1) = a_{HT}(n) +1 \quad \hbox{ and}
\end{equation}
\begin{equation} \label{HH}
    a_{HH}(n+2) = a_{HH}(n+1) + a_{HH}(n)
\end{equation}
for all $ n\ge 1$, with $a_{W}(1) = 0$ and $a_{W}(2) = 1$.
{\footnote{ For (\ref{HT}), the sequence $a_{HT}$ is given by $\{0,1,2,3,4,\dots\}$ or A001477 in \cite{OEIS}. There, the connection of A001477 to $a_{HT}$ is not  yet stated.  An alternative way to describe the sequence $a_{HT}(n)$ is given by the recursive formula $a_{HT}(n+1) = 2a_{HT}(n) - a_{HT}(n-1)$. This follows directly from  (\ref{HT}) since $2a_{HT}(n) - a_{HT}(n-1) = 2(a_{HT}(n-1)+1) - a_{HT}(n-1) = a_{HT}(n-1)+2 = a_{HT}(n)+1 = a_{HT}(n+1)$.  The fact that the sequence $a_{HH}$ is given by the Fibonacci numbers $\{0,1,1,2,3,5,\dots\}$ was established by Matthew Leingang \cite{HH}; see also A000045 in \cite{OEIS}. There, it is noted that $a_{HH}(n)$ is the ``number of binary sequences of length n-2 that have no consecutive 0's."}} 
\end{theorem}

\begin{proof}
(\ref{HT}): The statements $a_{HT}(1)=0$ and $a_{HT}(2)=1$ are obvious. To examine $a_{HT}(n+1)$, let $\mathcal{M}_{n+1}$ denote the set of sequences of $H$'s and $T$'s of length $n+1$ that contain the word $HT$ only at the end. Then, $a_{HT}(n+1) = |\mathcal{M}_{n+1}|$. We denote the  sequences in $\mathcal{M}_{n+1}$ that start with a $T$ by $\mathcal{M}^{T}_{n+1}$ and those that start with an $H$ by $\mathcal{M}^{H}_{n+1}$. Then,  $\mathcal{M}_{n+1}$ is the disjoint union of $\mathcal{M}^{H}_{n+1}$ and $ \mathcal{M}^{T}_{n+1}$.  Thus, $$a_{HT}(n+1) = |\mathcal{M}_{n+1}| = |\mathcal{M}^{T}_{n+1}| + |\mathcal{M}^{H}_{n+1}|.$$ 
If $x\in \mathcal{M}^{H}_{n+1}$, then $x$ must be the sequence $HHHH{\ldots}HT$. Thus, $|\mathcal{M}^{H}_{n+1}| = 1$.  If $x\in \mathcal{M}^{T}_{n+1}$, then there exists a unique $y\in \mathcal{M}_{n}$ such that $x=Ty$.  Define a map $\Phi : \mathcal{M}^{T}_{n+1} \to \mathcal{M}_{n}$ by $\Phi(x) := y$. Then, $\Phi$ is one-to-one and onto. Thus, $ |\mathcal{M}^{T}_{n+1}| = |\mathcal{M}_{n}|= a_{HT}(n)$, or
$$a_{HT}(n+1) = |\mathcal{M}^{T}_{n+1}| + |\mathcal{M}^{H}_{n+1}|= a_{HT}(n) +1.$$

\medskip
\noindent
(\ref{HH}): The statements $a_{HH}(1)=0$ and $a_{HH}(2)=1$ are obvious. To examine $a_{HH}(n)$, let $\mathcal{M}_n$ denote the set of sequences of $H$'s and $T$'s of length $n$ that contain the word $HH$ only at the end. Then, $a_{HH}(n) = |\mathcal{M}_n|$. Let $\N_n$ the set of sequences of length $n$ that end with an $T$ and do not contain the word $HH$. Now, if $x\in\M_{n+2}$, then there exists a unique $y\in\N_n$ such that $x = yHH$. Define $\Phi: \M_{n+2} \to \N_n$ by $\Phi(x) := y$. Then, $\Phi$ is one-to-one and onto. Therefore, 
$$ a_{HH}(n+2) = |\mathcal{M}_{n+2}| = |\N_n|.$$
Let $\N_n^{TT}, \N_n^{HT}$ be the sets of sequences in $\N_n$ that end with $TT$ or $HT$, respectively. Then, $\N_n$ is the disjoint union of  $\N_n^{TT}$ and $\N_n^{HT}$. Thus, 
$$|\N_n| = |\N_n^{TT}| + |\N_n^{HT}|.$$
If $x\in \N_n^{TT}$, then there exists a unique $y\in \N_{n-1}$ such that $x=yT$. Define $\Phi: \N_n^{TT} \to \N_{n-1}$ by $\Phi(x) := y$. Then, $\Phi$ is one-to-one and onto. Therefore,  
$$|\N_n^{TT}|= |\N_{n-1}| = |\mathcal{M}_{n+1}| = a_{HH}(n+1) .$$
If $x\in \N_n^{HT}$, then $x$ must end with the word $THT$ since $x$ cannot contain the word $HH$ before the last $T$. That is, for every  $x\in \N_n^{HT}$, there exists a unique $y\in\N_{n-2}$ such that $x=yHT$.  Define $\Phi: \N_n^{HT} \to \N_{n-2}$ by $\Phi(x) := y$. Then, $\Phi$ is one-to-one and onto. Therefore,  
$$|\N_n^{HT}|= |\N_{n-2}| = |\mathcal{M}_{n}| = a_{HH}(n) .$$
The equalities above imply that 
$a_{HH}(n+2) = a_{HH}(n+1) + a_{HH}(n)$ for all $n\ge 1$ with $a_{HH}(1) = 0$ and  $a_{HH}(2) = 1$.

\end{proof}

\begin{corollary} If one flips a fair coin,  then $a_{HT}(n) = n-1$  for all $n\ge 1$. The expected value of flips it takes for the word $HT$ to appear for the first time is
$\mu_{HT} = 4$ with standard deviation $\sigma_{HT} = 2$.\footnote{Observe that this is different from \cite{HT} where the expected number of flips required to observe at least one head and one tail is discussed.} Moreover, $$p_{HT}(\ge N) = \frac{a_{HT}(N+1)}{2^{N-1}} = \frac{N}{2^{N-1}} .$$ In particular, $p_{HT}(\ge 7) = 7/64 \approx 0.11$.
\end{corollary}

\begin{proof}
Since $a_{HT}(n) = a_{HT}(n-1) +1$ and $a_{HT}(1)=0$, it follows that $a_{HT}(n)=n-1$ for all $n\ge 1$.
Let $$h(x) := \sum_{n=1}^m x^n = \frac{1-x^{m+1}}{1-x} -1.$$
Then, $xh'(x) = \sum_{n=1}^m n\, x^n$. Therefore, for $|x|<1$ and $m\to\infty$, the generating functions $f_{HT,m}(x)$ and $f_{HT}(x)$ are given by
\begin{align*} f_{HT,m}(x) &= \sum_{n=1}^m (n-1) x^n = \sum_{n=1}^m n\, x^n - \sum_{n=1}^m x^n = xh'(x) - h(x) \\ &= \frac{x^2}{(1-x)^2}\left[ 1 - x^{m-1}(m(1-x)+x)\right] \\ & \to f_{HT}(x) = \sum_{n=1}^\infty (n-1) x^n = \frac{x^2}{(1-x)^2}. 
\end{align*}
It follows that
 $f'_{HT}(x) = \frac{2x}{(1-x)^3}$ and   $f''_{HT}(x) = \frac{4x+2}{(1-x)^4}$.
Since $R(f_{HT}) =1$, Proposition \ref{Prop1} yields $\mu_w = \frac{1}{2}f'_{HT}(\frac{1}{2}) = 4$, $\hbox{Var}_W = \frac{1}{4}f''_W(\frac{1}{2}) + \mu_W -  \mu_W^2 = 4$, and $p_{HT}(> m) = 1-p_W(\le m) = 1- f_{HT,m}(\frac{1}{2}) = (m+1)/2^m.$

\end{proof}

\begin{corollary} If one flips a fair coin,  then, for $n\ge 1$,
\begin{align} \label{fib}
a_{HH}(n) = \hbox{\rm Round}\left[ \frac{1}{\sqrt{5}}\left(\frac{1+\sqrt{5}}{2}\right)^{n-1} \right]
\end{align}
is the Fibonacci sequence $a_{HH} = \{0,1,1,2,3,5,8,13,21,\dots\}$.  The  expected value of flips it takes for the word $HH$ to appear for the first time is
$\mu_{HH} = 6$ 
with standard deviation  $\sigma_{HH} = \sqrt{22} \approx 4.7$. Moreover, 
$$p_{HH}(\ge N) = \frac{a_{HH}(N+2)}{2^{N-1}}.$$ 
In particular, $p_{HH}(\ge 12) \approx 0.1$.
\end{corollary}	

\begin{proof} To simplify writing, define $a_n:= a_{HH}(n)$ and $$f_m(x) := f_{HH,m}(x)= \sum_{n=1}^m a_{HH}(n) x^n = \sum_{n=1}^m a_nx^n.$$    Then,   $f_{HH}(x) = \sum_{n=1}^\infty a_{HH}(n) x^n = \lim_{m\to\infty} f_m(x)$ exists for $|x|<\frac{1}{2}$. By (\ref{HH}) and  $a_1=a_{HH}(1)=0$, it follows that 
\begin{align*}
    f_m(x) &= x^2 + \left( a_2+ a_1\right) x^3 + \left( a_3+ a_2\right) x^4 + \dots + (a_{m-1}+a_{m-2})x^m \\ &= x^2 + \left( a_2x^3 +  a_3x^4 + \cdots  +a_{m-1}x^m\right)\\
    & \qquad \qquad  + \left( a_2x^4 +  a_3x^5 + \cdots + +a_{m-2}x^m\right)\\
    &= x^2 + x\left(f_m(x)- a_mx^m\right) + x^2\left(f_m(x) - a_{m-1}x^{m-1} - a_mx^m\right)\\
    &= xf_m(x)+ x^2f_m(x) +x^2\left(1-a_mx^{m-1} - a_{m-1}x^{m-1}-a_mx^m \right)\\
    &= xf_m(x)+ x^2f_m(x) +x^2\left(1-a_{m+1}x^{m-1}-a_mx^m \right).
    \end{align*}
 This implies that
 \begin{align*} f_{HH,m}(x) = f_m(x) = \frac{-x^2}{x^2+x-1}\left[ 1 -a_{HH}(m+1)x^{m-1} - a_{HH}(m)x^{m}\right].
\end{align*}
Taking the limit as $m\to\infty$ yields 
\begin{align*} f_{HH}(x) = \frac{-x^2}{x^2+x-1} = -1 + \frac{x-1}{x^2+x-1} = -1 + \frac{x-1}{(x_1-x)(x_2-x)}
\end{align*}
for 
$x_1 = (-1+\sqrt{5})/{2}$ and $x_2 = (-1-\sqrt{5})/{2}$.
Using partial fractions, the geometric series representation of 
$$\frac{1}{x_i-x} = \frac{1}{x_i}\left(\frac{1}{1-\frac{x}{x_i}}\right) = \sum_{n=0}^\infty \left(\frac{1}{x_i}\right)^{n+1} x^n,$$ 
and a substantial amount of straightforward algebra, one obtains Binet's Formula (see also \cite{Binet}); namely,
$$ f_{HH}(x) = \sum_{n=1}^\infty \frac{1}{\sqrt{5}}\left[ \left(\frac{1+\sqrt{5}}{2}\right)^{n-1}- \left(\frac{1-
\sqrt{5}}{2}\right)^{n-1}\right] x^n =
\sum_{n=1}^\infty a_{HH}(n) x^n.$$ Now, (\ref{fib}) follows from the fact that $a_{HH}(n)\in \mathrm{N}$ and $\left| \frac{1}{\sqrt{5}}\left(\frac{1-\sqrt{5}}{2}\right)^{n-1}  \right| < \frac{1}{2}$ for all $n\ge 1$. Since $f_{HH}(x) = \frac{-x^2}{x^2+x-1}$, it follows that
$$ f'_{HH}(x) = \frac{-x(x-2)}{(x^2+x-1)^2} \hbox{ , } f''_{HH}(x) = \frac{2(x^3-3x^2-1)}{(x^2+x-1)^3}  ,$$
and $R(f_{HH}) = x_1 = (-1+\sqrt{5})/{2} >\frac{1}{2}$. By Proposition \ref{Prop1}, $\mu_{HH} = \frac{1}{2}f'_{HH}(\frac{1}{2}) = 6$ (see also \cite{HH}), $\sigma_{HH} = \sqrt{\frac{1}{4}f''_{HH}(\frac{1}{2}) + \mu_{HH} -  \mu_{HH}^2} =  \sqrt{22}$,
and 
\begin{align*}
p_{HH}(> m) & = 1-p_{HH}(\le m) = 1- f_{HH,m}(\frac{1}{2}) \\ &= a_{HH}(m+1)\left(\frac{1}{2}\right)^{m-1} + a_{HH}(m)\left(\frac{1}{2}\right)^{m}\\
&=\left[2a_{HH}(m+1)+ a_{HH}(m)\right]\left(\frac{1}{2}\right)^{m}
= a_{HH}(m+3)\left(\frac{1}{2}\right)^{m}.
\end{align*}

\end{proof} 

\section{Three-Letter Words}
\begin{theorem} When flipping a coin n times, let $W$ be one of the words $HHH, HTH, HTT, HHT$\footnote{The words $TTT$, $THH$, $TTH$, and $THT$ can be obtained by replacing all $H$'s with $T$'s and vice versa and the table  yields the corresponding recursion formulas.}, and let $a_{W}(n)$ be the number of ways that the word $W$ can appear for the first time at the $n$-th coin toss.
Then, $a_{W}(1) = a_{W}(2) = 0$, $a_{W}(3) = 1$, and
$$a_{W}(n+3) = Aa_{W}(n+2) + Ba_{W}(n+1)+ C a_{W}(n),$$ where the numbers $A,B,C$ as well as the first 15 terms of the sequences $a_W(n)$ are summarized in the following table.{\footnote{In \cite{OEIS}, the sequences for $HHH$, $HTT$, and $HTH$ are A000073, A000071, and A005314. There, the equivalent property of $a_{HHH}(n)$---that it gives the ``number of binary sequences of length n-3 that have no three consecutive 0's"---is noted. It also mentions that $a_{HTT}(n) = a_{HHT}(n)$ has the equivalent property that it gives the ``number of 001-avoiding binary words of length n - 3." Additionally, \cite{OEIS} touches on an equivalent property of $a_{HTH}(n)$, stating that it is ``the number of binary words of length n that begin with 1 and avoid the subword 101." }}
\smallskip

\begin{center}
\begin{tabular}{ |c|c|c|c|c| } 
 \hline
  W & A, B, C & Sequence $a_W(n)$ for $1\le n\le 15$  \\ 
  \hline
 HHH & 1, ${}$ 1, ${}$ 1    & 0, 0, 1, 1, 2, 4, ${}$ 7, 13, 24, 44, 81, 149, 274, 504, 927\\ 
 HTT & 2, ${}$ 0, -1 & 0, 0, 1, 2, 4, 7, 12, 20, 33, 54, 88, 143, 232, 376, 609 \\ 
 HHT & 2, ${}$ 0, -1 & 0, 0, 1, 2, 4, 7, 12, 20, 33, 54, 88, 143, 232, 376, 609 \\ 
 HTH & 2, -1, ${}$ 1 & 0, 0, 1, 2, 3, 5, ${}$ 9, 16, 28, 49, 86, 151, 265, 465, 816 \\
 
 \hline
\end{tabular}
\end{center}
\end{theorem}

\smallskip
\begin{proof} Clearly, $a_{W}(1) = 0$, $a_{W}(2) = 0$, and $a_{W}(3) = 1$ for all three-letter words $W$. We now prove the four cases $HHH, HTH, HTT, HHT$ separately. In parts (I) - (III) below, note that $\N_n$ denotes different sets.

\smallskip \noindent
(I) $W = HHH$. Consider $a_n :=a_{HHH}(n)$ for $n>3$. Let $\mathcal{M}_n$ denote the set of sequences of $H$'s and $T$'s of length $n$ that contain the word $HHH$ only at the end. Then, $a_n = |\mathcal{M}_n|$. Let $\N_n$ be the set of sequences of length $n$ that end with a $T$ and do not contain the word $HHH$. Now, if $x\in\M_{n+3}$, then there exists a unique $y\in\N_n$ such that $x = yHHH$. Define $\Phi: \M_{n+3} \to \N_n$ by $\Phi(x) := y$. Then, $\Phi$ is one-to-one and onto. Therefore, 
$$a_{n+3} = |\mathcal{M}_{n+3}| = |\N_n|.$$
Let $\N_n^{TT}, \N_n^{THT}, \N_n^{HHT}$ be the sets of sequences in $\N_n$ that end with $TT$, $THT$ or $HHT$, respectively. Then, $\N_n$ is the disjoint union of  $\N_n^{TT}$, $\N_n^{THT}$ and $\N_n^{HHT}$. Thus, $$|\N_n| = |\N_n^{TT}| + |\N_n^{THT}| + |\N_n^{HHT}|.$$
If $x\in \N_n^{TT}$, then there exists a unique $y\in N_{n-1}$ such that $x=yT$. Define $\Phi: \N_n^{TT} \to \N_{n-1}$ by $\Phi(x) := y$. Then, $\Phi$ is one-to-one and onto. Therefore,
$$|\N_n^{TT}|= |\N_{n-1}| = |\mathcal{M}_{n+2}| = a_{n+2} .$$
If $x\in \N_n^{THT}$, then there exists a unique $y\in\N_{n-2}$ such that $x=yHT$.  Define $\Phi: \N_n^{HT} \to \N_{n-2}$ by $\Phi(x) := y$. Then, $\Phi$ is one-to-one and onto. Therefore,  
$$|\N_n^{HT}|= |\N_{n-2}| = |\mathcal{M}_{n+1}| = a_{n+1} .$$
If $x\in \N_n^{HHT}$, then $x$ must end with the word $THHT$ since $x$ cannot contain the word $HHH$ before the last $T$. That is, for every  $x\in \N_n^{HHT}$,  there exists a unique $y\in\N_{n-3}$ such that $x=yHHT$.  Define $\Phi: \N_n^{HHT} \to \N_{n-3}$ by $\Phi(x) := y$. Then, $\Phi$ is one-to-one and onto. Therefore,  
$$|\N_n^{HT}|= |\N_{n-3}| = |\mathcal{M}_{n}| = a_n .$$
The equalities above imply that 
$a_{n+3} = a_{n+2} + a_{n+1} + a_n$ for all $n\ge 1$.

\smallskip \noindent
(II) $W = HTT$ or $W = HHT$. Let $\mathcal{M}_n$ denote the set of sequences of $H$'s and $T$'s of length $n$ that contain the word $W$ only at the end. Then, $a_n:= a_{W}(n) = |\mathcal{M}_n|$. For each $x\in \mathcal{M}_{n-1}$, one creates two new sequences of length $n$ by placing an $H$ or a $T$ at the beginning of $x$. The set of these sequences is denoted by $\mathcal{L}_{n}$. Note that $\mathcal{L}_{n}$ contains $2a_{n-1}$ sequences, all ending with the word  $W$; i.e., $|\mathcal{L}_{n}|=2a_{n-1}$. Since $\mathcal{M}_{n} \subseteq \mathcal{L}_{n}$, we consider  $\mathcal{N}_{n} :=  \mathcal{L}_{n}\setminus \mathcal{M}_{n}$ and note that $\mathcal{L}_n = \M_n\, \dot\cup \, \N_n$ (disjoint union). Therefore, 
	 $$2a_{n-1}= |\mathcal{L}_n| = |\M_n| + |\N_n| = a_n + |\N_n|.$$ 
	 Every sequence $x\in \mathcal{N}_n$ contains at least one occurrence of the word $W$ before the word $W$ at the end of the sequence. However, since all sequences in $\N_n$  are derived from adding a letter to the beginning of every sequence in $\mathcal{M}_{n-1}$, each $x\in\N_n$ must contain the word $W$ as its first three letters. Thus, each $x\in\N_n$ must begin with the word $W$ and end with it, with no occurrences of the word $W$ in between; i.e., for each $x\in\N_n$ there exists a unique $y\in \M_{n-3}$ such that $x= Wy$. Define $\Phi: \N_n \to \M_{n-3}$ by $\Phi(x):= y$. Then, $\Phi$ is one-to-one and onto. Therefore, $|\N_n| = |\M_{n-3}| = a_{n-3}$, and  $ 2a_{n-1} = a_n + a_{n-3}$. This shows that  $a_n = 2a_{n-1} -a_{n-3}$.

	 \smallskip \noindent
(III) $W = HTH$.
Let $\mathcal{M}_n$ denote the set of sequences of $H$'s and $T$'s of length $n$ that contain the word $HTH$ only at the end. Then, $a_n:=a_{HTH}(n) = |\mathcal{M}_n|$. For each $x\in \mathcal{M}_{n-1}$, one creates two new sequences of length $n$ by placing an $H$ or a $T$ at the beginning of $x$. The set of these sequences is denoted by $\mathcal{L}_{n}$. Note that $\mathcal{L}_{n}$ contains $2a_{n-1}$ sequences, all ending with $HTH$; i.e., $|\mathcal{L}_{n}|=2a_{n-1}$. Since $\mathcal{M}_{n} \subseteq \mathcal{L}_{n}$, we consider  $\mathcal{N}_{n} :=  \mathcal{L}_{n}\setminus \mathcal{M}_{n}$ and note that $\mathcal{L}_n = \M_n\, \dot\cup \, \N_n$ (disjoint union). Therefore, 
	 \begin{equation*}
	 2a_{n-1}= |\mathcal{L}_n| = |\M_n| + |\N_n| = a_n + |\N_n|.
	 \end{equation*}
	 Every sequence $x\in \mathcal{N}_n$ contains at least one occurrence of the word $HTH$ before the $HTH$ at the end of the sequence. However, since all sequences in $\N_n$  are derived from adding a letter to the beginning of every sequence in $\mathcal{M}_{n-1}$, each $x\in\N_n$ must contain the word $HTH$ as its first three letters. Thus, each $x\in\N_n$ must begin with $HTH$ and end with it, with no occurrences of the word $HTH$ in between; i.e., for each $x\in\N_n$ there exists a unique $y\in \M_{n-3}$ such that $x= HTHy$. This does not imply an onto map between $\mathcal{N}_n$ and $\mathcal{M}_{n-3}$, as $y$ can not begin with the two-letter sequence $TH$. However, let us denote the set $\mathcal{M}_{n-2}^H$ as the set of $n - 2$ length sequences where $HTH$ appears only at the end and which begin with $H$. Then, we can construct a one-to-one and onto map $\Phi:\mathcal{N}_n \to \mathcal{M}_{n-2}^H$ defined by $\Phi(HTHy) = Hy$ for every $y \in \mathcal{M}_{n-3}$. Now we will verify that $\mathcal{M}_{n-2}^H = a_{n-2} - a_{n-3}$. Indeed, if we also denote the sets of $n - 2$ length sequences where $HTH$ appears only at the end and which begin with $T$ as $\mathcal{M}_{n-2}^T$, then it is clear that $a_{n-2} = |\mathcal{M}_{n-2}^H| + |\mathcal{M}_{n-2}^T|$. We can construct a map $\Phi': \mathcal{M}_{n-2}^T \to \mathcal{M}_{n-3}$ defined by $\Phi'(Ty) = y$ for $y \in \mathcal{M}_{n-3}$ which is clearly one-to-one and onto. As such, this implies $|\mathcal{M}_{n-2}^T| = a_{n-3}$. Hence,
	 $ a_{n-2} = |\mathcal{M}_{n-2}^H| + |\mathcal{M}_{n-2}^T| = |\mathcal{M}_{n-2}^H| + a_{n-3}$, or
	 \begin{equation*}
	     |\mathcal{M}_{n-2}^H| = a_{HTH}(n-2) - a_{HTH}(n-3).
	 \end{equation*}
	 Since we know that $|\mathcal{N}_{n}| = |\mathcal{M}_{n-2}^H|$, we have 
	 \begin{equation*}
	     2a_{n-1} = |\mathcal{M}_n| + |\mathcal{N}_n| = a_n + a_{n-2} - a_{n-3}.
	 \end{equation*}
	 Thus, $
	     a_n = 2a_{n-1} - a_{n-2} +a_{n-3}.$
\end{proof}
\goodbreak

\begin{proposition} \label{Rec3} Consider the recurrence relation 
$$a_n = Aa_{n-1} + Ba_{n-2} + Ca_{n-3} \hbox{ for }n\ge 4$$ 
with $a_1 = a_2 = 0$, and $a_3 = 1$. Then, the finite generating function $f_m(x) = \sum_{n=1}^m a_nx^n$ is given by 
$$ \frac{-x^3}{Cx^3+Bx^2 +Ax-1} \left(1-x^m\left(\frac{a_{m+1}}{x^{2}} + \frac{Ba_m + Ca_{m-1}}{x}+Ca_m\right)\right) .$$
Moreover, for $x$ with sufficiently small magnitude, the generating function $f(x) = \lim_{m\to\infty}f_m(x) $ is given by
\begin{align*}f(x) &= \frac{-x^3}{Cx^3+Bx^2 +Ax-1}=  \frac{-x^3}{C(x-a)(x-b)(x-c)}
= \sum_{n=1}^\infty a_nx^n,
\end{align*}
where $a,b,c$ are the three roots of $Cx^3+Bx^2 +Ax-1=0$ and
\begin{align}\label{a_n}
  a_{n} =  \frac{a^{2-n}(b-c)-b^{2-n}(a-c)+c^{2-n}(a-b)}{C(a-b)(a-c)(b-c)}.  
\end{align}
\end{proposition}

\smallskip
\begin{proof}
 Since $a_1 = 0$, $a_2 = 0$, and $a_3 = 1$, it follows that 
\begin{align*}
    f_m(x) &= x^3 + \left( Aa_3 + Ba_2+ Ca_1\right) x^4 + \left(Aa_4+ Ba_3+ Ca_2\right) x^5 +\cdots \\
    & \qquad \qquad + (Aa_{m-1}+Ba_{m-2}+ Ca_{m-3})x^m \\ 
    &= x^3 + A\left( a_3x^4 +  a_4x^5 + \cdots  +a_{m-1}x^m\right)\\
    & \qquad \qquad  + B\left( a_2x^4 +  a_3x^5 + \cdots + +a_{m-2}x^m\right)\\
    & \qquad \qquad \qquad  + C\left( a_1x^4 +  a_2x^5 + \cdots + +a_{m-3}x^m\right)\\
    &= x^3 + Ax\left(f_m(x)- a_mx^m\right) \\
    & \qquad \quad + Bx^2\left(f_m(x) - a_{m-1}x^{m-1} - a_mx^m\right)\\
     & \qquad \quad \qquad + Cx^3\left(f_m(x) - a_{m-2}x^{m-2}- a_{m-1}x^{m-1} - a_mx^m\right)
     \end{align*}
    \begin{align*}
    &= Axf_m(x)+ Bx^2f_m(x) + Cx^3f_m(x) +x^3\\
    & \qquad \qquad - \left[Aa_{m}+Ba_{m-1}+Ca_{m-2}\right]x^{m+1}\\
    &\qquad\qquad \qquad- \left[Ba_m + Ca_{m-1}\right]x^{m+2}-Ca_mx^{m+3} \\
    &= Axf_m(x)+ Bx^2f_m(x) + Cx^3f_m(x) \\
    & \qquad \qquad + x^3\left(1-x^m\left(\frac{a_{m+1}}{x^{2}} + \frac{Ba_m + Ca_{m-1}}{x}+Ca_m\right) \right).
    \end{align*}
 This implies that $f_m(x)$ is given by
 \begin{align*}  
 \frac{-x^3}{Cx^3+Bx^2 +Ax-1} \left(1-x^m\left(\frac{a_{m+1}}{x^{2}} + \frac{Ba_m + Ca_{m-1}}{x}+Ca_m\right)\right) .
\end{align*}
By (\ref{prob}), for $|x|<1/2$, 
$$f(x) = \sum_{n=1}^\infty a_nx^n = \lim_{m\to\infty} f_m(x) = \frac{-x^3}{Cx^3+Bx^2 +Ax-1}$$
The Mathematica\textsuperscript{\textregistered}
command
``$\hbox{SeriesCoefficient}[-x^3/((x - a)*(x - b)*(x - c)), \{x, 0, n\}]$" yields the explicit formula for $a_n$ as stated (\ref{a_n}).
\end{proof}
 
 \smallskip

 \begin{corollary}
 Let $W = HHH$. Then, $a_{W}(1) = a_{W}(2) = 0$, $a_{W}(3) = 1$ and, for all $n\ge 1$, $a_{W}(n+3) = a_{W}(n+2) + a_{W}(n+1)+ a_{W}(n)$ is given by
\begin{align} \label{HHH}
a_{W}(n) 
=\hbox{\rm{Round}}\left[ \frac{a^{2-n}}{(a-b)(a-c)} \right],
\end{align}
where $a,b,c$ are the three roots of $x^3+x^2+x-1=0$. Namely,
$ a =\frac{-1}{3}-\frac{2}{3K}+ \frac{K}{3} \approx 0.54.$, $
    b =\frac{1}{3}\left(-1+\frac{1+i\sqrt{3}}{K}-\frac{K(1-i\sqrt{3})}{2}\right) 
    \approx -0.77 + 1.12 i$, and 
    $c = \overline{b}$
    for $K:=(17+3\sqrt{33})^{1/3}$.
The expected value of flips it takes for the word $W=HHH$ to appear for the first time is
$\mu_{W} = 14$ 
with standard deviation  $\sigma_{W} = \sqrt{142} \approx 11.9$.  Moreover, 
\begin{align*}
p_{W}(\ge N) = \frac{2a_{W}(N+2)- a_{W}(N-1)}{2^{N-1}}.
\end{align*}
In particular, $p_{W}(\ge 30) \approx 0.1$.

\end{corollary}	

\smallskip
\begin{proof}
To simplify writing, let $W = HHH$, $a_n:= a_{W}(n)$, and $$f_m(x) := f_{W,m}(x)= \sum_{n=1}^m a_{W}(n) x^n = \sum_{n=1}^m a_nx^n.$$ By Proposition \ref{Rec3}, it follows that 
\begin{align*} f_{m}(x) = \frac{-x^3}{x^3+x^2+x-1}\left[ 1 -R_m(x)\right],
\end{align*}
where $R_m(x) = x^m\left(\frac{a_{m+1}}{x^{2}} + \frac{a_m + a_{m-1}}{x}+a_m\right),$
and 
\begin{align*}  f_{W}(x) = \lim_{m\to\infty}f_m(x) =  \frac{-x^3}{x^3+x^2+x-1}.
\end{align*}
It follows that
$ f'_{W}(x) = -\frac{x^2(x-1)(x+3)}{(x^3+x^2+x-1)^2}$, and$$ 
f''_{W}(x) = \frac{(2x+2x^2)(x^4 +2x^3-8x^2+6x-3)}{(x^3+x^2+x-1)^3}  .$$
By using the Mathematica\textsuperscript{\textregistered} command
``$\hbox{Solve}[x^3 + x^2 + x - 1 == 0, x]$,"
one obtains $x^3 + x^2 + x - 1 = (x-a)(x-b)(x-c)$ with $a, b, c$ as stated above.  Now, (\ref{HHH}) follows from the fact that $a_{W}(n)\in\mathbb{N}$  and that, by (\ref{a_n}), 
$$ a_{W}(n) =  a_n = \frac{a^{2-n}}{(a-b)(a-c)} + R(n), $$ where  
$R(n) := \frac{b^{2-n}(c-a)+c^{2-n}(a-b)}{(a-b)(a-c)}$ has absolute value less than $ \frac{1}{2}$  for all $n\ge 1$.
Finally, it follows from (\ref{HHH}) that $R(f_{W}) >1/2$. Thus, by Proposition \ref{Prop1}, 
$\mu_{W} = \frac{1}{2}f'_{W}(\frac{1}{2}) = 14$ (see also \cite{HHH}),
$\sigma_{W} = \sqrt{\hbox{Var}_{W}} = \sqrt{\frac{1}{4}f''_{W}(\frac{1}{2}) + \mu_{W}^{} -  \mu_{W}^2} = \sqrt{142}$,
and 
\begin{align*}
p_{W}(> m) & = 1-p_{W}(\le m) = 1- f_{W,m}(\frac{1}{2}) = R_m\left(\frac{1}{2}\right) \\
&= \left(\frac{1}{2}\right)^m\left(4a_{m+1} + 3a_m + 2a_{m-1}\right)
= \left(\frac{1}{2}\right)^m\left(2a_{m+3} -a_m \right).
\end{align*}
\end{proof}

\begin{corollary}
Let $W = HHT$ or $W=HTT$. Then, $a_{W}(1) = a_{W}(2) = 0$, $a_{W}(3) = 1$ and, for all $n\ge 1$, $a_{W}(n+3) = 2a_{W}(n+2) - a_{W}(n)$ is given by the ``Fibonacci-Minus-One" sequence $ \{0,0,1,2,4,7,12,20,33,54,\dots\}$, where 
\begin{align} \label{HHTa}
a_W(n) = \hbox{\rm Round}\left[ \frac{1}{\sqrt{5}}\left(\frac{1+\sqrt{5}}{2}\right)^{n} \right]-1
\end{align}

\smallskip\noindent
The expected value of flips it takes for the word $W$ to appear for the first time is
$\mu_{W} = 8$ 
with standard deviation  $\sigma_{W} =\sqrt{24} \approx 4.9$; the probability of obtaining the word $W$ for the first time at or past the $N$-th flip is given by
\begin{align*}
p_W(\ge N) = \frac{2a_W(N+1)- a_W(N-1)}{2^{N-1}}.
\end{align*}
In particular, $p_{W}(\ge 15) \approx 0.1$.

\end{corollary}	

\smallskip
\begin{proof}
To simplify writing, let $W = HTT$ or $W=HHT$, $a_n:= a_{W}(n)$, and $$f_m(x) := f_{W,m}(x)= \sum_{n=1}^m a_{W}(n) x^n = \sum_{n=1}^m a_nx^n.$$ By Proposition \ref{Rec3}, one obtains 
$$f_{m}(x) = \frac{-x^3}{-x^3+2x-1}\left[ 1 -R_m(x)\right],$$ where
$R_m(x) = x^m\left(\frac{a_{m+1}}{x^{2}} - \frac{a_{m-1}}{x}-a_m\right)$ and 
$$f_{W}(x) = \lim_{m\to\infty}f_m(x) =  \frac{-x^3}{-x^3+2x-1}.$$
Then,
$ f'_{W}(x) = \frac{x^2(3-4x)}{(-x^3+2x-1)^2}$, and  $f''_{W}(x) = \frac{-2x(6x^4-6x^3+4x^2-6x+3)}{(-x^3+2x-1)^3}.$
By using the Mathematica\textsuperscript{\textregistered} command
``$\hbox{Solve}[-x^3 + 2x - 1 == 0, x],$''
one finds the decomposition $-x^3 + 2x - 1 = -(x-a)(x-b)(x-c)$ with 
$a = 1$, $
    b  = (-1-\sqrt{5})/2$, and $
    c = (-1+\sqrt{5})/{2}$. By (\ref{a_n}),
$$ a_{W}(n) =  \frac{b^{2-n}(a-c)+c^{2-n}(b-a)}{(a-b)(a-c)(b-c)} + R(n), $$ where 
$R(n) := \frac{-a^{2-n}}{(a-b)(a-c)} = - 1$ for all $n\ge 1$.
With some algebra,
\begin{align*}
a_{W}(n) &= \frac{1}{\sqrt{5}}\left[ \left(\frac{1+\sqrt{5}}{2}\right)^{n}- \left(\frac{1-
\sqrt{5}}{2}\right)^{n} \right]-1\\
&= \hbox{\rm Round}\left[ \frac{1}{\sqrt{5}}\left(\frac{1+\sqrt{5}}{2}\right)^{n} \right]-1.
\end{align*}
Finally, it follows from (\ref{HHTa}) that $R(f_{W}) >1/2$. Thus, by Proposition \ref{Prop1}, 
$\mu_{W} = \frac{1}{2}f'_{W}(\frac{1}{2}) = 8$ (see also \cite{HTT}), 
$\sigma_{W} = \sqrt{\frac{1}{4}f''_{W}(\frac{1}{2}) + \mu_{W}^{} -  \mu_{W}^2} = \sqrt{24}$, 
and 
\begin{align*}
p_{W}(> m)&= 1-p_{W}(\le m) = 1- f_{W,m}(\frac{1}{2}) = R_m\left(\frac{1}{2}\right) \\
&= \left(\frac{1}{2}\right)^m\left(4a_{m+1} - a_m - 2a_{m-1}\right)
= \left(\frac{1}{2}\right)^m\left(2a_{m+2} -a_m \right)
\end{align*}
\end{proof}

\begin{corollary}  Let $W = HTH$. Then, $a_{W}(1) = a_{W}(2) = 0$, $a_{W}(3) = 1$ and, $a_{W}(n+3) = 2a_{W}(n+2) - a_{W}(n+1)+ a_{W}(n)$. Moreover, for $n\ge 3$, 
\begin{equation} \label{HTHa}
a_{W}(n) 
=\hbox{\rm{Round}}\left[ \frac{a^{2-n}}{(a-b)(a-c)} \right],
\end{equation}
where $a,b,c$ are the three roots of $x^3-x^2+2x-1=0$. Namely,
\begin{equation*}
        a = \frac{1}{3}-\frac{5}{3K} +\frac{K}{3}\approx 0.57, \quad
    b = \frac{1}{3} - \frac{(1+i\sqrt{3})K}{6} + \frac{5(1-i\sqrt{3}}{6K}), \quad
    c = \overline{b},
\end{equation*}
where $K=\left(\frac{11+3\sqrt{69}}{2}\right)^{1/3}$.
The expected value of flips it takes for the word $W=HTH$ to appear for the first time is
$\mu_{W} = 10$
with standard deviation  $\sigma_{W} = \sqrt{58} \approx 7.6$.  Moreover, 
\begin{align*}
p_{W}(\ge N) = \frac{6a_{W}(N)- a_{W}(N-1)}{2^{N-1}}.
\end{align*}
In particular, $p_{W}(\ge 22) \approx 0.1$.
\end{corollary}
\begin{proof}
To simplify writing, let $W = HTH$, $a_n:= a_{W}(n)$, and $$f_m(x) := f_{W,m}(x)= \sum_{n=1}^m a_{W}(n) x^n = \sum_{n=1}^m a_nx^n.$$ By Proposition \ref{Rec3}, it follows that 
\begin{align*}  f_{m}(x) = \frac{-x^3}{x^3-x^2+2x-1}\left[ 1 -R_m(x)\right],
\end{align*}
where $R_m(x) = x^m\left(\frac{a_{m+1}}{x^{2}} + \frac{-a_m + a_{m-1}}{x}+a_m\right),$
and 
\begin{align*}  f_{W}(x) = \lim_{m\to\infty}f_m(x) =  \frac{-x^3}{x^3-x^2+2x-1}.
\end{align*}
It follows that
$ f'_{W}(x) = \frac{x^2(x-1)(x-3)}{(x^3-x^2+2x-1)^2}$  and  $$f''_{W}(x) = -\frac{2x(x^5-6x^4 +6x^3+3x^2-6x+3)}{(x^3-x^2+2x-1)^3}  .$$
By using the Mathematica\textsuperscript{\textregistered} command
``$\hbox{Solve}[x^3 -x^2 + 2x - 1 == 0, x]$'',
one obtains the decomposition $x^3 -x^2 + 2x - 1 = (x-a)(x-b)(x-c)$ with $a, b, c$ as stated in the corollary. The equality (\ref{HTHa}) follows from (\ref{a_n}), $a_n=a_W(n)\in\mathbb{N}$, and the fact that, for $n\ge 3$, 
\begin{align*}
 \left|\frac{-b^{2-n}(a-c)+c^{2-n}(a-b)}{C(a-b)(a-c)(b-c)} \right| < \frac{1}{2}.  
\end{align*}
Finally, it follows from (\ref{HTHa}) that $R(f_{W}) >1/2$. Thus, by Proposition \ref{Prop1},
$
    \mu_W = \frac{1}{2}f_{W}'\left(\frac{1}{2}\right) = 10
$,  $
    \sigma_{W} = \sqrt{\hbox{Var}_{W}} = \sqrt{\frac{1}{4}f''_{W}\left(\frac{1}{2}\right) + \mu_{W}^{} -  \mu_{W}^2} = \sqrt{58}
$, and 
\begin{align*}
    p(> m) & = 1 - p(\leq m) = 1 - f_{HTH,m}\left(\frac{1}{2}\right)
     = R_m\left(\frac{1}{2}\right)\\
     &= \left(\frac{1}{2}\right)^m\left(4a_{m+1}+-2a_m + 2a_{m+1}+a_m\right)
    = \left(\frac{1}{2}\right)^m(6a_{m+1}-a_m).
\end{align*}
\end{proof}

\end{document}